\documentclass[12pt,reqno]{amsart}

\usepackage{amssymb,amsmath,graphicx,amsfonts,euscript,mathrsfs}
\usepackage{xcolor}
\usepackage{hyperref}

\numberwithin{equation}{section}

\def\H{{\cal H}}

\def\R{\mathbb{R}}

\def\F{\overrightarrow{F}}
\def\q{\overrightarrow{q}}
\def\H1{H^1}
\def\N{\mathcal{N}}

\DeclareMathOperator{\sech}{sech} 

\newtheorem{thm}{Theorem}
\newtheorem{lem}{Lemma}

\newtheorem{prop}{Proposition}

\newtheorem{cor}{Corollary}

\newtheorem{defn}{Definition}

\newtheorem{remark}{Remark}

\makeatletter
\newcommand{\Extend}[5]{\ext@arrow0099{\arrowfill@#1#2#3}{#4}{#5}}
\makeatother
\RequirePackage{color}\definecolor{RED}{rgb}{1,0,0}\definecolor{BLUE}{rgb}{0,0,1} 
\DeclareOldFontCommand{\sf}{\normalfont\sffamily}{\mathsf} 
\providecommand{\DIFaddtex}[1]{{\protect\color{red} #1}} 
\providecommand{\DIFaddbegin}{} 
\providecommand{\DIFaddend}{} 
\providecommand{\DIFdelbegin}{} 
\providecommand{\DIFdelend}{} 
\providecommand{\DIFaddbeginFL}{} 
\providecommand{\DIFaddendFL}{} 
\providecommand{\DIFdelbeginFL}{} 
\providecommand{\DIFdelendFL}{} 
\providecommand{\DIFadd}[1]{\texorpdfstring{\DIFaddtex{#1}}{#1}} 
\newcommand{\DIFscaledelfig}{0.5}
\RequirePackage{settobox} 
\RequirePackage{letltxmacro} 
\newsavebox{\DIFdelgraphicsbox} 
\newlength{\DIFdelgraphicswidth} 
\newlength{\DIFdelgraphicsheight} 
\LetLtxMacro{\DIFOincludegraphics}{\includegraphics} 
\newcommand{\DIFaddincludegraphics}[2][]{{\color{red}\fbox{\DIFOincludegraphics[#1]{#2}}}} 
\newcommand{\DIFdelincludegraphics}[2][]{
\sbox{\DIFdelgraphicsbox}{\DIFOincludegraphics[#1]{#2}}
\settoboxwidth{\DIFdelgraphicswidth}{\DIFdelgraphicsbox} 
\settoboxtotalheight{\DIFdelgraphicsheight}{\DIFdelgraphicsbox} 
\scalebox{\DIFscaledelfig}{
\parbox[b]{\DIFdelgraphicswidth}{\usebox{\DIFdelgraphicsbox}\\[-\baselineskip] \rule{\DIFdelgraphicswidth}{0em}}\llap{\resizebox{\DIFdelgraphicswidth}{\DIFdelgraphicsheight}{
\setlength{\unitlength}{\DIFdelgraphicswidth}
\begin{picture}(1,1)
\thicklines\linethickness{2pt} 
{\color[rgb]{1,0,0}\put(0,0){\framebox(1,1){}}}
{\color[rgb]{1,0,0}\put(0,0){\line( 1,1){1}}}
{\color[rgb]{1,0,0}\put(0,1){\line(1,-1){1}}}
\end{picture}
}\hspace*{3pt}}} 
} 
\LetLtxMacro{\DIFOaddbegin}{\DIFaddbegin} 
\LetLtxMacro{\DIFOaddend}{\DIFaddend} 
\LetLtxMacro{\DIFOdelbegin}{\DIFdelbegin} 
\LetLtxMacro{\DIFOdelend}{\DIFdelend} 
\DeclareRobustCommand{\DIFaddbegin}{\DIFOaddbegin \let\includegraphics\DIFaddincludegraphics} 
\DeclareRobustCommand{\DIFaddend}{\DIFOaddend \let\includegraphics\DIFOincludegraphics} 
\DeclareRobustCommand{\DIFdelbegin}{\DIFOdelbegin \let\includegraphics\DIFdelincludegraphics} 
\DeclareRobustCommand{\DIFdelend}{\DIFOaddend \let\includegraphics\DIFOincludegraphics} 
\LetLtxMacro{\DIFOaddbeginFL}{\DIFaddbeginFL} 
\LetLtxMacro{\DIFOaddendFL}{\DIFaddendFL} 
\LetLtxMacro{\DIFOdelbeginFL}{\DIFdelbeginFL} 
\LetLtxMacro{\DIFOdelendFL}{\DIFdelendFL} 
\DeclareRobustCommand{\DIFaddbeginFL}{\DIFOaddbeginFL \let\includegraphics\DIFaddincludegraphics} 
\DeclareRobustCommand{\DIFaddendFL}{\DIFOaddendFL \let\includegraphics\DIFOincludegraphics} 
\DeclareRobustCommand{\DIFdelbeginFL}{\DIFOdelbeginFL \let\includegraphics\DIFdelincludegraphics} 
\DeclareRobustCommand{\DIFdelendFL}{\DIFOaddendFL \let\includegraphics\DIFOincludegraphics} 
\RequirePackage{listings} 
\RequirePackage{color} 
\lstdefinelanguage{DIFcode}{ 
  moredelim=[il][\color{red}\scriptsize]{\%DIF\ <\ }, 
  moredelim=[il][\color{blue}\sffamily]{\%DIF\ >\ } 
} 
\lstdefinestyle{DIFverbatimstyle}{ 
	language=DIFcode, 
	basicstyle=\ttfamily, 
	columns=fullflexible, 
	keepspaces=true 
} 
\lstnewenvironment{DIFverbatim}{\lstset{style=DIFverbatimstyle}}{} 
\lstnewenvironment{DIFverbatim*}{\lstset{style=DIFverbatimstyle,showspaces=true}}{} 

\begin{document}

\setcounter{page}{1}

\title[Endpoint instability for gBBM]{Orbital instability of solitary waves for the generalized BBM equation at the negative critical endpoint}

\author{Cui Ning}
\address{School of Financial Mathematics and Statistics, Guangdong University of Finance, Guangzhou 510521, China}
\email{cuining@gduf.edu.cn}
\thanks{}
%
%
%




\begin{abstract}\noindent
We study the orbital stability of solitary waves to the
generalized Benjamin--Bona--Mahony equation
$$
u_t+u_x+\kappa\frac{p+1}{2}(u^p)_x-u_{txx}=0,
\qquad (t,x)\in\mathbb R^+\times\mathbb R,
$$
where $\kappa\in\mathbb R\setminus\{0\}$ and $p\geq2$ is an integer. This equation
admits solitary waves of the form
$$
u(t,x)=\phi_c(x-ct).
$$
At the negative critical endpoint speed
$$
c=c_p^-=
\frac{p-1}{2(p+1)}
\left(1-\sqrt{\frac{p+3}{2}}\right)<0,
$$
we prove that the corresponding solitary wave is orbitally
unstable whenever either $\kappa<0$, or $\kappa>0$ and $p$ is even.
The proof combines a codimension-two coercivity estimate, two-parameter
modulation, and a corrected localized virial functional adapted to the
degeneracy of the momentum slope at the endpoint.
\end{abstract}
 \maketitle
\section{Introduction}
We study orbital instability  of solitary waves for the generalized Benjamin--Bona--Mahony (gBBM) equation
  \begin{equation}\label{pbbm}
u_t+u_x+\kappa\frac{p+1}{2}(u^p)_x-u_{txx}=0,
\qquad (t,x)\in\mathbb R^+\times\mathbb R,
\end{equation}
where $\kappa\in\mathbb R\setminus\{0\}$ and $p\ge2$ is an  integer. The BBM equation was introduced as a regularized model for unidirectional long waves, and equations of BBM type have subsequently been used in several long-wave  settings; see, for example, \cite{benbon72,per66}. The Cauchy problem and the qualitative behavior of solitary waves for BBM-type equations have been studied extensively; see \cite{albohe87,sostr90,zeng03}  and the references therein .

Equation \eqref{pbbm} possesses the conserved quantities
\[
E(u)=\frac{1}{2}\int_{\mathbb R}(u^2+\kappa u^{p+1})\,dx,
\qquad
Q(u)=\frac{1}{2}\int_{\mathbb R}(u^2+u_x^2)\,dx,
\]
and solitary-wave solutions $u(t,x)=\phi_c(x-ct)$,  where $\phi_c$  solves
\begin{equation}\label{phie}
c\phi_c''+(1-c)\phi_c+\frac{\kappa(p+1)}{2}\phi_c^p=0.
\end{equation}
For the parameter regimes considered below, an explicit profile is
\begin{equation}\label{phic}
\phi_c(x)=A_c\,\sech^{\frac{2}{p-1}}(K_cx),
\end{equation}
where $K_c=\frac{p-1}{2}\sqrt{\frac{c-1}{c}}$ and $A_c=\left(\frac{c-1}{\kappa}\right)^\frac{1}{p-1}$.
Thus negative-speed solitary waves exist when  either $\kappa<0$, or when $\kappa>0$ and $p$ is even. Here $A_c$ denotes the real-root convention, which is crucial on the negative branch and will be adopted throughout this paper.

The stability theory of solitary waves is governed by the variational action $S_c=E-cQ$ and by the spectral properties of $S_c''(\phi_c)$. The general Hamiltonian framework of Grillakis--Shatah--Strauss \cite{grishastr87,gss90}, together with BBM-specific variational and instability arguments of Bona and collaborators and of Souganidis--Strauss \cite{bona75,albohe87,sostr90,bomcre00,zeng03}, gives the usual nondegenerate criterion in terms of the slope of the momentum along the solitary-wave branch. Pego and Weinstein \cite{pegwei92} further related the sign change of this slope to spectral instability for solitary waves of dispersive equations.

For the normalization in \eqref{pbbm}, the two critical speeds are
\begin{equation}\label{c_p}
c_p^+=\frac{p-1}{2(p+1)}\left(1+\sqrt{\frac{p+3}{2}}\right)>0,
\qquad
c_p^-=\frac{p-1}{2(p+1)}\left(1-\sqrt{\frac{p+3}{2}}\right)<0.
\end{equation}
On the positive branch  $c>1$  with  $\kappa>0$, the standard slope criterion yields stability for $2\le p\le5$ and, for $p>5$,   stability  for $c>c_p^+$ and  instability for $1<c<c_p^+$. The critical positive-speed case is degenerate because $\partial_cQ(\phi_c)=0$; in the equivalent power-law notation used by Jia and Wu, orbital instability at this positive critical speed was proved in \cite{JW25}. On the negative branch, the non-endpoint theory gives stability for $c<c_p^-$ and instability for $c_p^-<c<0$ in the admissible cases considered here; see \cite{kal06,kalngu09,kal09,ngkal092} and related BBM stability literature.

The purpose of this paper is to settle the remaining negative-speed endpoint $c=c_p^-$.   At this speed,
\[
\partial_cQ(\phi_c)\big|_{c=c_p^-}=0,
\]
so  the standard nondegenerate slope criterion does not decide stability. The degeneracy also affects the modulation argument: besides the translation kernel of $S_c''(\phi_c)$, variation of the speed produces a second distinguished direction, and the usual rough estimate $|\dot y-\lambda|\lesssim\|\xi\|_{H^1}$ leaves a first-order term in the localized virial identity.

Our argument has three main ingredients. First, we construct an explicit direction
  \[
\Gamma_c=c^3\partial_c\phi_c-\frac{c^2}{p-1}\phi_c+\frac{c}{2}x\partial_x\phi_c+c\phi_c
\]
for which $\langle S_c''(\phi_c)\Gamma_c,\Gamma_c\rangle<0$, and set  $\tau_c=S_c''(\phi_c)\Gamma_c$.
This produces a coercivity estimate on the codimension-two space orthogonal to $\partial_x\phi_c$  and $\tau_c$. Second, we use these two orthogonality conditions to modulate both translation and speed. Third, a localized primitive of the fixed endpoint direction $\Gamma_c$ gives a refined identity for $\dot y-\lambda$ with constants uniform in the localization scale; a fixed-coefficient correction to the localized virial functional then cancels the linear modulation term up to quadratic errors. The resulting corrected virial quantity has  a strictly positive derivative for a suitable one-parameter  family of perturbations,   which yields orbital instability.

\begin{defn}
The solitary wave  $\phi_c(x-ct)$  is said to be orbitally stable in $H^1(\mathbb R)$ if, for every $\varepsilon>0$, there exists $\delta>0$ such that
\[
\inf_{s\in\mathbb R}\|u_0-\phi_c(\cdot-s)\|_{H^1}<\delta
\]
implies that the corresponding solution  is global and satisfies
\[
\inf_{s\in\mathbb R}\|u(t)-\phi_c(\cdot-s)\|_{H^1}<\varepsilon
\]
for all $t\geq 0$. Otherwise the solitary wave  is said to be orbitally  unstable.
\end{defn}

We now state the main result of this paper.

\begin{thm}\label{thm:mainthm}
Let $p\geq2$ be an integer and let $c=c_p^-$ . Assume that either
 $\kappa<0$, or
 $\kappa>0$ and  $p$ is  even.
Then the solitary wave $\phi_c(x-ct)$ of \eqref{pbbm} is  orbitally unstable.
\end{thm}

This paper is organized as follows. In Section 2, we introduce some important functionals and collect several useful lemmas. In Section 3, we construct the negative direction and establish the coercivity estimate. In Section 4, we establish the modulation estimates. In Section 5, we prove Theorem \ref{thm:mainthm}.

\section{Preliminaries}

\subsection{Notations}
We use $X\lesssim Y$ to denote an estimate of the form
$X\leq CY$ for some constant $C>0$. Similarly, we will write $X\sim Y$ to mean $X\lesssim Y$ and $Y\lesssim X$.

For $u,v\in L^2(\mathbb{R})$, we define
$$(u,v)_{L^2}=\int_{\mathbb{R}}u(x){v(x)}\,dx$$
and regard $L^2(\mathbb{R})$ as a real Hilbert space.

For a function $f(x)$, its $L^{q}(\R)$-norm $\|f\|_{q}=\Big(\displaystyle\int_{\mathbb{R}} |f(x)|^{q}dx\Big)^{\frac{1}{q}}$
and its $\H1(\R)$-norm $\|f\|_{H^1}=(\|f\|^2_{2}+\|\partial_x f\|^2_{2})^{\frac{1}{2}}$.

For $\varepsilon>0$, we denote the orbital neighborhood of $\phi_c$ by
\begin{equation*}
U_\varepsilon(\phi_c)
:=
\left\{
u\in H^1(\mathbb R):
\inf_{y\in\mathbb R}
\|u-\phi_c(\cdot-y)\|_{H^1}<\varepsilon
\right\}.
\end{equation*}

\subsection{Conservation laws}
The solution $u(t)$ of  \eqref{pbbm} satisfies the following conservation laws
$${E}(u(t))={E}(u_0),  {Q}(u(t))={Q}(u_0)$$
for all $t\in[0,T_{max})$, where
\begin{align*}
{E}(u(t))&=\frac{1}{2}\int_{\R}(u^2+\kappa u^{p+1})\, dx,\\
{Q}(u(t))&=\frac{1}{2}\int_{\R}(u^2+u_x^2)\,dx.
\end{align*}

\subsection{Some functionals}
From the definitions of $E$, and $Q$, we have
\begin{align}
E'(u)=&u+\frac{\kappa(p+1)}{2}u^p,\label{E'}\\
Q'(u)=&u-\partial^2_xu\label{Q'}.
\end{align}
Now we define
\begin{align*}
S_c(u)&=E(u)-cQ(u)\\
&=\frac{1}{2}\int_{\R}(u^2+\kappa u^{p+1})\, dx-\frac{c}{2}\int_{\R}(u^2+u_x^2)\, dx.
\end{align*}
Then we have
\begin{align}
S'_c(u)&=E'(u)-cQ'(u)\nonumber\\
&=c\partial_x^2u+(1-c)u+\frac{\kappa(p+1)}{2}u^p.\label{S'}
\end{align}
Moreover, \eqref{phie} is equivalent to $S'_c(\phi_c)=0$.

\subsection{Useful Lemmas}
In this subsection, we give some lemmas which are useful in the following sections.
First, we have the following  identities.
\begin{lem}\label{S''}
The operator $S''_c(\phi_c)$ is self-adjoint, that is, for any $f,g\in\H1(\R)$,
\begin{align}\label{26.1}
\langle S''_c(\phi_c)f,g\rangle=\langle S''_c(\phi_c)g,f\rangle.
\end{align}
Furthermore,
  \begin{align}
&S''_c(\phi_c)\partial_x\phi_c=0,\label{SCP}\\
&S''_c(\phi_c)\phi_c=\frac{\kappa(p+1)(p-1)}{2}\phi_c^p,\nonumber\\
&S''_c(\phi_c)(x\partial_x\phi_c)=2c\partial_x^2\phi_c,\nonumber\\
&S''_c(\phi_c)\partial_c\phi_c=Q'(\phi_c)\nonumber.
\end{align}
\end{lem}
\begin{proof}
Since
$$
\partial_t\partial_sS_c(\phi_c+sg+tf)=\partial_s\partial_tS_c(\phi_c+sg+tf),
$$
setting $t=s=0$ yields \eqref{26.1}.
Moreover, by \eqref{S'}, we have
\begin{align}\label{def:S''}
S''_c(\phi_c)f=c\partial_x^2f+(1-c)f+\frac{\kappa(p+1)p}{2}\phi_c^{p-1}f.
\end{align}
The remaining formulas are obtained by combining the above identity with a straightforward computation.
\end{proof}

\begin{lem}\label{S-phi-phi}
Let $c=c_p^-<0$. Then we have
$$\langle S''_c(\phi_c)\phi_c,\phi_c\rangle <0.$$
\end{lem}
\begin{proof}
From \eqref{def:S''}, we have
$$
\langle S''_c(\phi_c)\phi_c, \phi_c\rangle=\frac{\kappa(p+1)(p-1)}{2}\int_{\R}\phi_c^{p+1}\,dx.
$$
We consider two cases.

Case (i): $c=c_p^-<0$, $\kappa<0$. By \eqref{phic}, we know that $\phi_c>0$ and $\frac{\kappa(p+1)(p-1)}{2}<0$. Hence we have $\langle S''_c(\phi_c)\phi_c, \phi_c\rangle<0$.

Case (ii): $c=c_p^-<0$, $\kappa>0$ with $p$ is even. By \eqref{phic}, we obtain $\phi_c<0$ and $\frac{\kappa(p+1)(p-1)}{2}>0$. Since $p$ is even, we also get $\langle S''_c(\phi_c)\phi_c, \phi_c\rangle<0$.

This finishes the proof.
\end{proof}

\begin{lem}\label{connection}
The solitary wave $\phi_c$ satisfies the following identities:
  \begin{align*}
    &\|\partial_x\phi_c\|_2^2=\frac{(p-1)(c-1)}{(p+3)c}\|\phi_c\|_2^2,\\
    &\int_{\R}\phi_c^{p+1}dx=\frac{4(c-1)}{\kappa(p+3)}\|\phi_c\|_2^2,\\
    &\partial_c\|\phi_c\|_2^2=\frac{4c-p+1}{2(p-1)(c-1)c}\|\phi_c\|_2^2.
  \end{align*}
\end{lem}
\begin{proof}
Taking the inner products of \eqref{phie} with $\phi_c$ and with $x\partial_x\phi_c$,  respectively, and integrating by parts, we obtain
\begin{align*}
&c\|\partial_x\phi_c\|_2^2+(1-c)\|\phi_c\|_2^2+\frac{\kappa(p+1)}{2}\int_{\R}\phi_c^{p+1}\,dx=0,\\
&\frac{c}{2}\|\partial_x\phi_c\|_2^2-\frac{c-1}{2}\|\phi_c\|_2^2+\frac{\kappa}{2}\int_{\R}\phi_c^{p+1}\,dx=0
\end{align*}
Solving the above two equations yields
\begin{align}
&\|\partial_x\phi_c\|_2^2=\frac{(p-1)(c-1)}{(p+3)c}\|\phi_c\|_2^2,\label{xphi}\\
&\int_{\R}\phi_c^{p+1}dx=\frac{4(c-1)}{\kappa(p+3)}\|\phi_c\|_2^2.\label{pphi}
\end{align}
Next, by a scaling argument, we obtain
$$\|\phi_c\|_2^2=\left[\frac{2(c-1)}{\kappa(p+1)}\right]^{\frac{2}{p-1}}\sqrt{\frac{c-1}{c}}
\|\phi_0\|_2^2,$$
where $\phi_0$ is the solution of
$$-\partial_x^2\phi_0+\phi_0-\phi_0^p=0.$$
Differentiating with respect to $c$, we finally obtain
\begin{align}\label{cphi}
\partial_c\|\phi_c\|_2^2=\frac{4c-p+1}{2(p-1)(c-1)c}\|\phi_c\|_2^2.
\end{align}
This completes the proof of the lemma.
\end{proof}

\begin{lem}\label{Qc}
Let $c=c_p^-$, we have
 $$\partial_cQ(\phi_c)=0.$$
\end{lem}

\begin{proof}
From the definition of $Q$, we have
\begin{align*}
Q(\phi_c)=\frac{1}{2}\|\phi_c\|_2^2+\frac{1}{2}\|\partial_x\phi_c\|_2^2.
\end{align*}
Substituting the expression for $\|\partial_x\phi_c\|_2^2$ from Lemma \ref{connection} yields
\begin{align*}
Q(\phi_c)=\frac{1}{2}\left[1+\frac{(p-1)(c-1)}{(p+3)c}\right]\|\phi_c\|_2^2
=\frac{2(p+1)c-(p-1)}{2(p+3)c}\|\phi_c\|_2^2.
\end{align*}
Differentiating with respect to $c$ gives
\begin{align*}
  \partial_cQ(\phi_c)=\frac{2(p+1)c-(p-1)}{2(p+3)c}\partial_c\|\phi_c\|_2^2
  +\frac{p-1}{2(p+3)c^2}\|\phi_c\|_2^2.
\end{align*}
Inserting the expression for $\partial_c\|\phi_c\|_2^2$ from Lemma \ref{connection} and simplifying, we obtain
  \begin{align*}
\partial_cQ(\phi_c)&=\frac{8(p+1)c^2-8(p-1)c-(p-1)^2}{4(p+3)(p-1)(c-1)c^2}\|\phi_c\|_2^2\\
&=\frac{2(p+1)(c-c_p^-)(c-c_p^+)}{(p+3)(p-1)(c-1)c^2}\|\phi_c\|_2^2,
\end{align*}
where $c_p^-$ and $c_p^+$ are defined in \eqref{c_p}.

Therefore, $\partial_cQ(\phi_c)=0$ at $c=c_p^-$, which completes the proof of the lemma.
\end{proof}

We next turn to an investigation of the spectrum of the self-adjoint operator $S_c''(\phi_c)$.
\begin{lem}\label{L5}
Let $c=c_p^-$ and assume one of the parameter regimes in Theorem
\ref{thm:mainthm}. Then
$S_c''(\phi_c)$ has exactly one negative eigenvalue, and
$$\mathrm{ker} S_c''(\phi_c)=\mathrm{Span}\{\partial_x\phi_c\}.$$
\end{lem}
\begin{proof}
Set $w=c^{-1}<0$ and define $\psi_w=\theta\phi_c$ with $\theta=[\frac{\kappa(p+1)}{2c}]^{\frac{1}{p-1}}$.
Substituting  into \eqref{phie} shows that $\psi_\omega$ satisfies
\begin{align}\label{l-w}
-\partial_x^2\psi_w+(1-w)\psi_w-\psi_w^p=0.
\end{align}

Define the operator $L_wf=-\partial_x^2f+(1-w)f-p\psi_w^{p-1}f$. Differentiating \eqref{l-w} with respect to $w$ yields $L_w\partial_w\psi_w=\psi_w.$ A direct computation from the definitions gives
\begin{align}\label{s-l}
S_c''(\phi_c)f=-cL_wf.
\end{align}
It is known (see \cite{weinstein1985}) that $L_w$ has exactly one negative eigenvalue. By Lemma \ref{S-phi-phi}, $\langle S''_c(\phi_c)\phi_c,\phi_c\rangle<0$, so $S''_c(\phi_c)$ has at least one negative eigenvalue.

Since $c<0$, the factor $-c>0$ in \eqref{s-l} preserves the sign and multiplicity of eigenvalues. Consequently, $S_c''(\phi_c)$ has exactly one negative eigenvalue.

%

Next, we proceed to prove the property of $\ker S_c''(\phi_c)$. From Lemma \ref{S''}, $S''_c(\phi_c)\partial_x\phi_c=0$, so $\partial_x\phi_c \in \ker S_c''(\phi_c)$.

Conversely, suppose $S_c''(\phi_c)v = 0$. The Wronskian
$$W(x)=\partial_x^2\phi_cv-\partial_x\phi_c \partial_xv$$
satisfies $W'(x) = 0$. By Liouville's formula, hence $W(x)$ is constant. As $|x|\rightarrow \infty$, $\partial_x\phi_c$ and $\partial_x^2\phi_c$ decay to zero, so $W(x)\rightarrow 0$. Thus $W(x)\equiv0$, which implies that $v$ is a scalar multiple of $\partial_x\phi_c$. Therefore $\ker S_c''(\phi_c) \subseteq \mathrm{Span}\{\partial_x\phi_c\}$.

Combining $S''_c(\phi_c)\partial_x\phi_c=0$ gives
$\mathrm{Ker} S_c''(\phi_c) = \mathrm{Span}\{\partial_x\phi_c\}.$
This completes the proof of lemma.
\end{proof}


\begin{remark}\label{R1}
The spectral properties established in Lemma \ref{L5} imply the following orthogonal decomposition of $ L^2(\mathbb{R}) $:
$$
L^2(\mathbb{R}) = \operatorname{Span}\{\chi_{-1}\} \oplus \operatorname{Span}\{\partial_x\phi_c\} \oplus \mathcal{P},
$$
where $ \chi_{-1} $ is the eigenfunction corresponding to the unique negative eigenvalue  $\lambda_{-1}<0$, and $\mathcal{P} $ is the positive spectral subspace satisfying
$$
\langle S_c''(\phi_c) p, p \rangle \ge \rho_0 \|p\|_{L^2}^2, \qquad \forall p \in \mathcal{P},
$$
for some $ \rho_0 > 0 $. For convenience, we normalize the eigenfunction $\chi_{-1}$ such that $\|\chi_{-1}\|_{L^2}=1$.
\end{remark}

\begin{cor}
The following identity holds:
$$S''_c(\phi_c)(c\partial_c\phi_c-\frac{1}{p-1}\phi_c)=\phi_c.$$
\end{cor}
\begin{proof}
Recall that $\psi_w=\theta\phi_c$ with $\omega=c^{-1}$ and $\theta=[\frac{\kappa(p+1)}{2c}]^{\frac{1}{p-1}}$. A straightforward computation yields $$\frac{-1}{c\theta}\partial_w\psi_w=c\partial_c\phi_c-\frac{1}{p-1}\phi_c.$$
Using the relations $S_c''(\phi_c)f=-cL_wf$ and $L_w\partial_w\psi_w=\psi_w$, we obtain
\begin{align}
S''_c(\phi_c)(c\partial_c\phi_c-\frac{1}{p-1}\phi_c)=-cL_w[\frac{-1}{c\theta}\partial_w\psi_w]=\frac{1}{\theta}\psi_w=\phi_c.
\end{align}
The proof is complete.
\end{proof}

\section{Negative direction and Coercivity}
Define
$$\Gamma_c=c^3\partial_c\phi_c-\frac{c^2}{p-1}\phi_c+\frac{c}{2}x\partial_x\phi_c+c\phi_c,$$
and
$$\tau_c=S''_c(\phi_c)\Gamma_c=c^2\phi_c+c^2\partial_x^2\phi_c+\frac{c\kappa(p+1)(p-1)}{2}\phi_c^p.$$
The following lemma provides the required negative direction.
\begin{lem}\label{L6}
 Let $c=c_p^-$ and $p\geq2$. Then
  $$\langle S''_c(\phi_c)\Gamma_c, \Gamma_c\rangle<0.$$
\end{lem}
\begin{proof}
Substituting the definitions of $\Gamma_c$ and $\tau_c$, we obtain
\begin{align*}
\langle S''_c(\phi_c)&\Gamma_c, \Gamma_c\rangle=\langle \tau_c, \Gamma_c\rangle\\
&=\langle c^2\phi_c+c^2\partial_x^2\phi_c+\frac{c\kappa(p+1)(p-1)}{2}\phi_c^p, c^3\partial_c\phi_c-\frac{c^2}{p-1}\phi_c+\frac{c}{2}x\partial_x\phi_c+c\phi_c\rangle.
\end{align*}
Expanding the inner product, integrating by parts where appropriate, and using Lemma \ref{connection}, we obtain
\begin{align*}
 \langle S''_c(\phi_c)\Gamma_c, \Gamma_c\rangle
=-c^5\partial_cQ(\phi_c)+\frac{c^2 f(p,c)}{4(p+3)(p-1)(c-1)}\|\phi_c\|_2^2,
\end{align*}
where the polynomial obtained by the full expansion is
\begin{align*}
f(p,c)=&\,8(p+1)c^3+(8p^3-12p^2+16p+4)c^2\\
&\,-(p-1)(16p^2-11p+19)c+(8p-1)(p-1)^2.
\end{align*}
By Lemma \ref{Qc}, at $c=c_p^-$ we have $\partial_cQ(\phi_c)=0$ and
\begin{align}\label{zero}
8(p+1)c^2-8(p-1)c-(p-1)^2=0.
\end{align}
Using \eqref{zero} to eliminate the higher powers of $c$, the corrected polynomial reduces to
\[
f(p,c)=\frac{(p-1)(p+3)}{2(p+1)}
\left[-16c(p^2-p+1)+2p^3+5p^2-8p+1\right].
\]
Since $c=c_p^-<0$ and $p\ge2$, both $-16c(p^2-p+1)$ and
$2p^3+5p^2-8p+1$ are strictly positive. Hence $f(p,c)>0$.
Therefore,
\begin{align}\label{L61}
\langle S''_c(\phi_c)\Gamma_c, \Gamma_c\rangle=\frac{c^2 f(p,c)}{4(p+3)(p-1)(c-1)}\|\phi_c\|_2^2<0.
\end{align}

%
%
The proof is complete.
\end{proof}

We next establish the coercivity estimate for the second variation.
\begin{prop}(Coercivity)\label{coercivity}
There exists a constant $C_0$ such that for any $\eta\in H^1(\R)$ satisfying
$$\langle\eta,\partial_x\phi_c\rangle=\langle\eta,\tau_c\rangle=0,$$
we have
$$\langle S''_c(\phi_c)\eta, \eta\rangle\geq C_0\|\eta\|_{H^1}^2.$$
\end{prop}

\begin{proof}
By Remark \ref{R1}, $ L^2(\mathbb{R})$ admits the orthogonal decomposition
$$
L^2(\mathbb{R}) = \operatorname{Span}\{\chi_{-1}\} \oplus \operatorname{Span}\{\partial_x\phi_c\} \oplus \mathcal{P},
$$
where $\chi_{-1}$ is the eigenfunction corresponding to the unique negative eigenvalue $ \lambda_{-1}$, and $\mathcal{P}$ is the positive spectral subspace. We decompose $\eta$ and $\Gamma_c$ according to the orthogonal decomposition as follows:
\begin{align*}
\eta& = \alpha_\eta \chi_{-1} + \beta_\eta \partial_x \phi_c + p_\eta, \qquad \alpha_\eta\in \mathbb{R}, \ \beta_\eta \in \mathbb{R}, \  p_\eta \in \mathcal{P},\\
\Gamma_c &= \alpha_\Gamma \chi_{-1} + \beta_\Gamma \partial_x \phi_c + p_\Gamma, \qquad \alpha_\Gamma\in \mathbb{R}, \ \beta_\Gamma\in \mathbb{R}, p_\Gamma \in \mathcal{P}.
\end{align*}
Since $\langle \eta, \partial_x\phi_c \rangle = 0$, we have $\beta_\eta=0$. Therefore, we can write
\begin{align}\label{etad}
\eta=\alpha_\eta \chi_{-1} + p_\eta.
\end{align}
Moreover, we have
\begin{align}
\langle S_c''(\phi_c) \eta, \eta \rangle &= \langle S_c''(\phi_c)(\alpha_\eta \chi_{-1} + p_\eta), \alpha_\eta \chi_{-1} + p_\eta \rangle\nonumber\\
&=\lambda_{-1}\alpha_\eta^2+\langle S_c''(\phi_c)p_\eta, p_\eta \rangle\label{SEE}
\end{align}
By \eqref{SCP}, we have
\begin{align}\label{taud}
\tau_c=S_c''(\phi_c)\Gamma_c
=S_c''(\phi_c)(\alpha_\Gamma \chi_{-1} + \beta_\Gamma \partial_x \phi_c + p_\Gamma)
=S_c''(\phi_c)(\alpha_\Gamma \chi_{-1} + p_\Gamma).
\end{align}
Moreover, we have
\begin{align*}
\langle S_c''(\phi_c)\Gamma_c, \Gamma_c\rangle
&=\langle S_c''(\phi_c)(\alpha_\Gamma \chi_{-1} + p_\Gamma), \alpha_\Gamma \chi_{-1} + \beta_\Gamma \partial_x \phi_c + p_\Gamma\rangle\nonumber\\
&=\lambda_{-1}\alpha_\Gamma^2+\langle S_c''(\phi_c)p_\Gamma, p_\Gamma\rangle<0.
\end{align*}
In particular, $\alpha_\Gamma\neq0$. There exist a positive constant $\sigma_0>0$, such that
\begin{align}\label{LSS}
\lambda_{-1}\alpha_\Gamma^2+\langle S_c''(\phi_c)p_\Gamma, p_\Gamma\rangle=-\sigma_0<0.
\end{align}
Using $\langle \eta, \tau_c \rangle = 0$, \eqref{etad} and \eqref{taud}, we have
\begin{align*}
\langle \eta, \tau_c \rangle=\langle \alpha_\eta \chi_{-1} + p_\eta, S_c''(\phi_c)(\alpha_\Gamma \chi_{-1} + p_\Gamma)\rangle
=\lambda_{-1}\alpha_\eta\alpha_\Gamma+\langle S_c''(\phi_c)p_\eta, p_\Gamma\rangle
=0.
\end{align*}
Then, we have
\begin{align*}
\big(-\lambda_{-1}\alpha_\eta\alpha_\Gamma)\big)^2
=\big(\langle S_c''(\phi_c)p_\eta, p_\Gamma\rangle\big)^2.
\end{align*}
Moreover, we obtain
\begin{align*}
(-\lambda_{-1}\alpha^2_\eta)(-\lambda_{-1}\alpha^2_\Gamma)
=\big(\langle S_c''(\phi_c)p_\eta, p_\Gamma\rangle\big)^2
\leq \langle S_c''(\phi_c)p_\eta, p_\eta\rangle\langle S_c''(\phi_c)p_\Gamma, p_\Gamma\rangle.
\end{align*}
Combining with \eqref{LSS}, we know that
\begin{align}
\lambda_{-1}\alpha^2_\eta
&\geq \frac{\langle S_c''(\phi_c)p_\eta, p_\eta\rangle\langle S_c''(\phi_c)p_\Gamma, p_\Gamma\rangle}{\lambda_{-1}\alpha^2_\Gamma}\nonumber\\
&=\frac{\langle S_c''(\phi_c)p_\eta, p_\eta\rangle\langle S_c''(\phi_c)p_\Gamma, p_\Gamma\rangle}{-\big(\sigma_0+\langle S_c''(\phi_c)p_\Gamma, p_\Gamma\rangle\big)}.\label{LASS}
\end{align}
 Then, we have
\begin{align}\label{AE}
\alpha_\eta^2\leq  \frac{\langle S_c''(\phi_c)p_\eta, p_\eta\rangle\langle S_c''(\phi_c)p_\Gamma, p_\Gamma\rangle}{-\lambda_{-1}(\sigma_0+\langle S_c''(\phi_c)p_\Gamma, p_\Gamma\rangle}).
\end{align}
Using \eqref{SEE} and \eqref{LASS}, we obtain
\begin{align}
\langle S_c''(\phi_c) \eta, \eta \rangle\geq&
\frac{\langle S_c''(\phi_c)p_\eta, p_\eta\rangle\langle S_c''(\phi_c)p_\Gamma, p_\Gamma\rangle}{-\big(\sigma_0+\langle S_c''(\phi_c)p_\Gamma, p_\Gamma\rangle\big)}
+\langle S_c''(\phi_c)p_\eta, p_\eta\rangle\nonumber\\
=&\langle S_c''(\phi_c)p_\eta, p_\eta\rangle\big[1-\frac{\langle S_c''(\phi_c)p_\Gamma, p_\Gamma\rangle}{\sigma_0+\langle S_c''(\phi_c)p_\Gamma, p_\Gamma\rangle}\big]\nonumber\\
=&\langle S_c''(\phi_c)p_\eta, p_\eta\rangle\frac{\sigma_0}{\sigma_0+\langle S_c''(\phi_c)p_\Gamma, p_\Gamma\rangle}\label{SEEG}\\
\geq&\frac{\rho_0\sigma_0}{\sigma_0+\langle S_c''(\phi_c)p_\Gamma, p_\Gamma\rangle}\|p_\eta\|_2^2.\nonumber
\end{align}
By \eqref{etad}, \eqref{AE}, \eqref{SEEG}and $\langle S_c''(\phi_c)p_\Gamma, p_\Gamma \rangle\geq\rho_0\|p_\Gamma\|_2^2$, we have
\begin{align*}
\|\eta\|_2^2=&\|\alpha_\eta \chi_{-1} + p_\eta\|_2^2=\alpha_\eta^2+\|p_\eta\|_2^2\\
\leq&\frac{\langle S_c''(\phi_c)p_\eta, p_\eta\rangle\langle S_c''(\phi_c)p_\Gamma, p_\Gamma\rangle}{-\lambda_{-1}(\sigma_0+\langle S_c''(\phi_c)p_\Gamma, p_\Gamma\rangle)}+\|p_\eta\|_2^2\\
\leq&\frac{\langle S_c''(\phi_c)\eta,\eta\rangle\langle S_c''(\phi_c)p_\Gamma, p_\Gamma\rangle}{-\lambda_{-1}\sigma_0}+\|p_\eta\|_2^2\\
\lesssim&\langle S_c''(\phi_c)\eta,\eta\rangle.
\end{align*}
Hence, there exists a constant $C'>0$ such that
\[
\langle S_c''(\phi_c)\eta,\eta\rangle\ge C'\|\eta\|_2^2.
\]
From the explicit form of $ S_c''(\phi_c) $,
integration by parts gives
$$
\langle S_c''(\phi_c) \eta, \eta \rangle
= -c \|\partial_x\eta\|_{L^2}^2 + \int_{\mathbb{R}} V(x) \eta^2 \, dx,
$$
where
$$
V(x) := 1 - c + \kappa \frac{p(p+1)}{2} \phi_c^{p-1}(x).
$$

Since $ \phi_c $ is bounded and decays exponentially, $ V $ is bounded below; hence there exists $ C'' > 0 $ such that $ V(x) \ge -C'' $ for all $ x \in \mathbb{R} $. Thus
\begin{align}\label{tag2}
\langle S_c''(\phi_c) \eta, \eta \rangle \ge -c \|\eta_x\|_{L^2}^2 - C'' \|\eta\|_{L^2}^2.
\end{align}
Combining the preceding $L^2$ coercivity estimate with \eqref{tag2},  for any $ \theta \in (0,1) $,
\begin{align*}
\langle S_c''(\phi_c) \eta, \eta \rangle
&= \theta \langle S_c''(\phi_c) \eta, \eta \rangle + (1-\theta) \langle S_c''(\phi_c) \eta, \eta \rangle \\
&\ge -c\theta \|\eta_x\|_{L^2}^2 + \big( C'(1-\theta) - C'' \theta \big) \|\eta\|_{L^2}^2.
\end{align*}
Choose $ \theta $ sufficiently small so that
$$
C'(1-\theta) - C'' \theta > 0,
$$
for instance $ \theta = \frac{C'}{2(C'+C'')} $. Define
$$
C_0 := \min \left\{ -c\theta,\; C'(1-\theta) - C'' \theta \right\} > 0.
$$
Then
$$
\langle S_c''(\phi_c) \eta, \eta \rangle \ge C_0 \left( \|\eta_x\|_{L^2}^2 + \|\eta\|_{L^2}^2 \right) = C_0 \|\eta\|_{H^1}^2.
$$
The proof is complete.
\end{proof}
For $\lambda<0$ in a fixed sufficiently small neighborhood of $c$, define
\begin{equation}\label{Gamma-lambda-def}
\Gamma_\lambda
:=\lambda^3\partial_\lambda\phi_\lambda
-\frac{\lambda^2}{p-1}\phi_\lambda
+\frac{\lambda}{2}x\partial_x\phi_\lambda
+\lambda\phi_\lambda,
\qquad
\tau_\lambda:=S_\lambda''(\phi_\lambda)\Gamma_\lambda.
\end{equation}

\begin{cor}[Uniform coercivity]\label{uniform-coercivity}
There exist $\delta_0>0$ and $C_*>0$ such that, for every
$|\lambda-c|<\delta_0$ and every $\eta_*\in H^1(\mathbb R)$ satisfying
\[
\langle\eta_*,\partial_x\phi_\lambda\rangle
=
\langle\eta_*,\tau_\lambda\rangle=0,
\]
one has
\[
\langle S_\lambda''(\phi_\lambda)\eta_*,\eta_*\rangle
\ge C_*\|\eta_*\|_{H^1}^2.
\]
\end{cor}

\begin{proof}
The maps
\[
\lambda\mapsto S_\lambda''(\phi_\lambda),\qquad
\lambda\mapsto\partial_x\phi_\lambda,\qquad
\lambda\mapsto\tau_\lambda
\]
are continuous in the relevant operator and $H^1$ topologies.
Moreover, $\partial_x\phi_c$ is odd whereas $\tau_c$ is even, so
$\langle\partial_x\phi_c,\tau_c\rangle=0$.

Let $\eta_*$ satisfy the two orthogonality conditions at $\lambda$ and set
\[
a=\frac{\langle\eta_*,\partial_x\phi_c\rangle}
        {\|\partial_x\phi_c\|_2^2},
\qquad
b=\frac{\langle\eta_*,\tau_c\rangle}{\|\tau_c\|_2^2},
\qquad
\eta=\eta_*-a\partial_x\phi_c-b\tau_c.
\]
Then $\eta$ satisfies the two orthogonality conditions at
$c$. Since the corresponding inner products at $\lambda$ vanish,
continuity gives
\[
|a|+|b|
\le C|\lambda-c|\,\|\eta_*\|_{H^1},
\qquad
\|\eta_*-\eta\|_{H^1}
\le C|\lambda-c|\,\|\eta_*\|_{H^1}.
\]
Proposition \ref{coercivity} therefore yields
\[
\langle S_c''(\phi_c)\eta,\eta\rangle
\ge C_0\|\eta\|_{H^1}^2.
\]
Furthermore,
\[
\left|
\langle S_\lambda''(\phi_\lambda)\eta_*,\eta_*\rangle
-
\langle S_c''(\phi_c)\eta,\eta\rangle
\right|
\le C|\lambda-c|\,\|\eta_*\|_{H^1}^2.
\]
Taking $\delta_0$ sufficiently small proves the claim.
\end{proof}

\section{Modulation}
Throughout this section we fix $c=c_p^-$. Under the contradiction
hypothesis of orbital stability, for every sufficiently small
$\varepsilon>0$ the solution issued from sufficiently close initial
data remains in $U_\varepsilon(\phi_c)$ for all $t\ge0$. We therefore
assume below that
\[
u(t)\in U_\varepsilon(\phi_c),\qquad t\ge0,
\]
and decompose the solution by translation and variation of the wave
speed.

\begin{prop}\label{modulation}
 There exists $\varepsilon_0>0$ such that for any $\varepsilon\in(0,\varepsilon_0)$, if $u(t)\in U_\varepsilon(\phi_c)$ for any $t\in\geq 0$, then the following properties hold. There exist $C^1$ functions
$$(\lambda(t), y(t)):U_\varepsilon(\phi_{c}) \rightarrow \R^-\times\R$$
such that, upon defining the perturbation
\begin{align}
\xi(t,x)=u(t,x+y(t))-\phi_{\lambda(t)}(x),\label{modulation-u}
\end{align}
the following orthogonality conditions are satisfied for all $t\geq 0$:
\begin{align}\label{orth-condition}
\langle\xi(t),\partial_x\phi_{\lambda(t)}\rangle=\langle\xi(t),\tau_{\lambda(t)}\rangle=0,
\end{align}
where $\Gamma_\lambda$ and $\tau_\lambda$ are defined in
\eqref{Gamma-lambda-def}. Moreover,
\begin{equation}\label{modulation-close}
|\lambda(t)-c|+\|\xi(t)\|_{H^1}\le C\varepsilon,
\qquad t\ge0.
\end{equation}
For brevity, we omit the argument $t$ in $\lambda(t),\ y(t),\ \xi(t)$ in what follows.  Moreover, the time derivatives satisfy the estimate
\begin{align}\label{xit-yt}
\big|\dot\lambda\big|+\big|\dot y-\lambda\big|
=O\big( \|\xi\|_{H^1}\big).
\end{align}
\end{prop}
\begin{proof}
 Define $$\overrightarrow{q}:=(u;\lambda(t),y(t))$$
and
$$\overrightarrow{q_0}:=(\phi_c;c,0).$$

Let $\F(\q)=(F_1,F_2)$ with
$$F_1(\q):=\langle\xi, \partial_x\phi_\lambda\rangle,\quad F_2(\q):=\langle\xi, \tau_\lambda\rangle.$$
A direct computation at $\q=\q_0$ yields
\begin{align*}
  &\partial_\lambda F_1(\q)\Big|_{\q=\q_0}=0;\qquad \qquad\qquad \partial_y F_1(\q)\Big|_{\q=\q_0}=\|\partial_x\phi_c\|_2^2;\\
  &\partial_\lambda F_2(\q)\Big|_{\q=\q_0}=-\langle\partial_c\phi_c,\tau_c\rangle;\qquad
  \partial_y F_2(\q)\Big|_{\q=\q_0}=\langle\partial_x\phi_c,\tau_c\rangle.
\end{align*}
Consequently, the Jacobian matrix of $\F(\q)$ with respect to $(\lambda,y)$, evaluated at $\q=\q_0$, is given by
\begin{align*}
D\F(\q_0)&=\begin{pmatrix}
\partial_\lambda F_1&\partial_y F_1\\
\partial_\lambda F_2&\partial_y F_2
\end{pmatrix}\Biggl|_{\q=\q_0}\notag\\
&=\begin{pmatrix}
0&\|\partial_x\phi_c\|_2^2\\
-\langle\partial_c\phi_c,\tau_c\rangle&\langle\partial_x\phi_c,\tau_c\rangle
\end{pmatrix}.
\end{align*}
Indeed,  the determinant of the Jacobian matrix $D\F(\q_0)$ is given by
\begin{align*}
\det\big(D\F(\q_0)\big)=\|\partial_x\phi_c\|_2^2\langle\partial_c\phi_c,\tau_c\rangle\ne 0.
\end{align*}
Using the definition of $\tau_c$  together with  Lemma \ref{S''}, we compute
\begin{align*}
 \langle\partial_c\phi_c,\tau_c\rangle
 =&\langle\partial_c\phi_c,S''_c(\phi_c)\Gamma_c\rangle
 =\langle Q'(\phi_c),c^3\partial_c\phi_c-\frac{c^2}{p-1}\phi_c+\frac{c}{2}x\partial_x\phi_c+c\phi_c\rangle\\
 =&c^3\partial_cQ(\phi_c)+\langle \phi_c-\partial_x^2\phi_c,\frac{-c^2}{p-1}\phi_c+\frac{c}{2}x\partial_x\phi_c+c\phi_c\rangle.
\end{align*}
At $c=c_p^-$, Lemma \ref{Qc} gives $\partial_cQ(\phi_c)=0$. Thus, after  simplification,
$$\langle\partial_c\phi_c,\tau_c\rangle=\frac{8pc-6(p-1)}{4(p+3)}\|\phi_c\|_2^2\neq 0,$$
since $c=c_p^-<0$ and $p\geq 2$. Hence the determinant is nonzero.


The Implicit Function Theorem then ensures the existence of a unique pair of $C^1$ functions $(\lambda, y):\ U_\varepsilon(\phi_{c})\rightarrow\R^-\times\R$ such that \eqref{orth-condition} holds for all $t\geq0$. The quantitative form of the same argument gives \eqref{modulation-close}, uniformly in time as long as $u(t)\in U_\varepsilon(\phi_c)$.

We now turn to estimates on the derivatives. Rewriting $u$ as
$$
u(t,x)=(\phi_\lambda+\xi)(t,x-y(t)),
$$
and applying \eqref{pbbm}, we obtain
\begin{align}\label{ut1}
  u_t=\dot\xi+\dot\lambda\partial_\lambda\phi_\lambda-\dot y\partial_x(\phi_\lambda+\xi)=JE'(\phi_\lambda+\xi),
\end{align}
where $J=-(1-\partial_x^2)^{-1}\partial_x$. Adding $\lambda\partial_x(\phi_\lambda+\xi)$ to both sides of \eqref{ut1} yields
\begin{align*}
\dot\xi+\dot\lambda\partial_\lambda\phi_\lambda-(\dot y-\lambda)\partial_x(\phi_\lambda+\xi)
&=JE'(\phi_\lambda+\xi)+J\lambda\partial_x(\phi_\lambda+\xi)\\
&=J\left[E'(\phi_\lambda+\xi)-\lambda(1-\partial_x^2)(\phi_\lambda+\xi)\right].
\end{align*}
Since $Q'(\phi_\lambda+\xi)=(1-\partial_x^2)(\phi_\lambda+\xi)$, it follows that
\begin{align}\label{ut2}
\dot\xi+\dot\lambda\partial_\lambda\phi_\lambda-(\dot y-\lambda)\partial_x(\phi_\lambda+\xi)
&=J\left[E'(\phi_\lambda+\xi)-\lambda Q'(\phi_\lambda+\xi)\right]=JS'_\lambda(\phi_\lambda+\xi).
\end{align}
Using a Taylor expansion about $\phi_\lambda$ and recalling that $S'(\phi_\lambda)=0$, we have
\begin{align}\label{spx}
S'_\lambda(\phi_\lambda+\xi)=S'_\lambda(\phi_\lambda)+S''_\lambda(\phi_\lambda)\xi+O(\xi^2)
=S''_\lambda(\phi_\lambda)\xi+O(\xi^2).
\end{align}
Inserting \eqref{spx} into \eqref{ut2} gives
\begin{align}\label{ut3}
\dot\xi+\dot\lambda\partial_\lambda\phi_\lambda-(\dot y-\lambda)\partial_x(\phi_\lambda+\xi)
&=JS''_\lambda(\phi_\lambda)\xi+\mathcal{N}_1(\xi),
\end{align}
where the remainder $\N_1(\xi)$ satisfies
$$\langle \N_1(\xi), f\rangle=O(\|\xi\|_{\H1}^2\,\|f\|_{\H1}),\qquad f\in\H1(\R).$$
Taking the inner product of \eqref{ut3} with $\partial_x\phi_\lambda$ and $\tau_\lambda$, respectively, we obtain
\begin{align}
  \langle\dot\xi,\partial_x\phi_\lambda\rangle
  +\dot\lambda\langle\partial_\lambda\phi_\lambda,& \partial_x\phi_\lambda\rangle
-(\dot y-\lambda)\langle \partial_x(\phi_\lambda+\xi),\partial_x\phi_\lambda\rangle\nonumber\\
  &=\langle JS''_\lambda(\phi_\lambda)\xi, \partial_x\phi_\lambda\rangle+O(\|\xi\|_{\H1}^2),\label{ut-p}\\
\langle\dot\xi,\tau_\lambda\rangle
   +\dot\lambda\langle\partial_\lambda\phi_\lambda,\tau_\lambda\rangle
    -&(\dot y-\lambda)\langle \partial_x(\phi_\lambda+\xi),\tau_\lambda\rangle
=\langle JS''_\lambda(\phi_\lambda)\xi,\tau_\lambda\rangle+O(\|\xi\|_{\H1}^2).\label{ut-t}
\end{align}

Since $\phi_\lambda$ is even, it follows from its definition that  $\tau_\lambda$ is also even.
Using the orthogonality conditions \eqref{orth-condition}, we reduce \eqref{ut-p} and \eqref{ut-t} to
\begin{align}
-\dot \lambda\langle\xi,\partial_x\partial_\lambda\phi_\lambda\rangle
-(\dot y-\lambda)&\left[\|\partial_x\phi_\lambda\|_2^2-\langle\xi,\partial_x^2\phi_\lambda\rangle\right]\nonumber\\
&=\langle \xi, S''_\lambda(\phi_\lambda)(1-\partial_x^2)^{-1}\partial_x^2\phi_\lambda\rangle+O(\|\xi\|_{\H1}^2),\label{ut6}\\
\dot\lambda\left[-\langle \xi, \partial_\lambda\tau_\lambda\rangle+\langle\partial_\lambda\phi_\lambda, \tau_\lambda\rangle\right]&
+(\dot y-\lambda)\langle\xi, \partial_x\tau_\lambda\rangle\nonumber\\
&=\langle \xi, S''_\lambda(\phi_\lambda)(1-\partial_x^2)^{-1}\partial_x\tau_\lambda\rangle+O(\|\xi\|_{\H1}^2).\label{ut7}
\end{align}
Denote the matrix
\begin{align*}
A:=\begin{pmatrix}
-\langle\xi,\partial_x\partial_\lambda\phi_\lambda\rangle&
-\|\partial_x\phi_\lambda\|_2^2+\langle\xi,\partial_x^2\phi_\lambda\rangle\\
-\langle \xi, \partial_\lambda\tau_\lambda\rangle+\langle\partial_\lambda\phi_\lambda, \tau_\lambda\rangle&\langle\xi, \partial_x\tau_\lambda\rangle
\end{pmatrix}
\end{align*}
A direct computation then yields
\begin{align*}
\begin{pmatrix}
  \dot\lambda\\
  \dot y-\lambda
\end{pmatrix}
=A^{-1}
\begin{pmatrix}
 \langle \xi, S''_\lambda(\phi_\lambda)(1-\partial_x^2)^{-1}\partial_x^2\phi_\lambda\rangle\\
 \langle \xi, S''_\lambda(\phi_\lambda)(1-\partial_x^2)^{-1}\partial_x\tau_\lambda\rangle
\end{pmatrix}
+
\begin{pmatrix}
 O(\|\xi\|_{\H1}^2)\\
 O(\|\xi\|_{\H1}^2)
\end{pmatrix}
=\begin{pmatrix}
 O(\|\xi\|_{\H1})\\
 O(\|\xi\|_{\H1})
\end{pmatrix}.
\end{align*}
This finishes the proof of the proposition.
\end{proof}
Let $\rho\in C^\infty(\R)$ be a smooth cutoff
function satisfying
\begin{align*}
\rho(x)=
   \left\{ \aligned
    &1,\quad |x|\leq 1,\\
    &0,\quad |x|\geq 2.
   \endaligned
  \right.
\end{align*}
For $R>0$, set $\rho_R(x):=\rho(\frac{x}R)$. Define
\begin{align}\label{f}
  f_\lambda:=\rho_R(x)\int_0^x\Gamma_\lambda(z)\,dz,
\end{align}
where
$\Gamma_\lambda=\lambda^3\partial_\lambda\phi_\lambda-\frac{\lambda^2}{p-1}\phi_\lambda+\frac{\lambda}{2}x\partial_x\phi_\lambda+\lambda\phi_\lambda$.

We then have the following corollary.

\begin{cor}\label{bl}
Under the same assumptions as in Proposition \ref{modulation},  the following estimate holds:
\begin{align*}
  \dot y-\lambda=\frac{-1}{B(\lambda)}\partial_t\langle\xi, (1-\partial_x^2)f_\lambda\rangle
  +O(\|\xi\|_{\H1}^2+\frac{1}R),
\end{align*}
where
\[
B(\lambda):=\frac{8(p^2+1)\lambda^2-2(7p+1)(p-1)\lambda+5(p-1)^2}{4(p+3)(p-1)(\lambda-1)}\|\phi_\lambda\|_2^2.
\]

Moreover, there exist constants $\delta_B>0$ and $b_0>0$ such that
\begin{align}\label{B-lower}
|B(\lambda)|\ge b_0,
\qquad
|\lambda-c|<\delta_B.
\end{align}
\end{cor}

\begin{proof}
Taking the inner product of \eqref{ut3} with $ (1-\partial_x^2)f_\lambda$ and integrating by parts, we obtain
\begin{align}\label{ut-f}
\langle\dot\xi, (1-\partial_x^2)f_\lambda\rangle
+&\dot\lambda\langle\partial_\lambda\phi_\lambda,(1-\partial_x^2)f_\lambda\rangle
-(\dot y-\lambda)\langle \partial_x(\phi_\lambda+\xi),(1-\partial_x^2)f_\lambda\rangle \nonumber\\
&=\langle JS''_\lambda(\phi_\lambda)\xi, (1-\partial_x^2)f_\lambda\rangle+O(\|\xi\|_{\H1}^2).
\end{align}

Since $\phi_\lambda$ is exponentially decaying and even, it follows that $f_\lambda\in L^2(\R)$ and is odd. Consequently, $\langle\partial_\lambda\phi_\lambda,(1-\partial_x^2)f_\lambda\rangle=0$.

Furthermore, from the rough estimate $\dot\lambda=O(\|\xi\|_{\H1})$ in  Proposition \ref{modulation}, we have
\begin{align*}
\langle\dot\xi, (1-\partial_x^2)f_\lambda\rangle=&\partial_t\langle\xi, (1-\partial_x^2)f_\lambda\rangle
-\dot\lambda\langle\xi, (1-\partial_x^2)\partial_\lambda f_\lambda\rangle\nonumber\\
=&\partial_t\langle\xi, (1-\partial_x^2)f_\lambda\rangle+O(\|\xi\|_{\H1}^2).
\end{align*}
Substituting this into \eqref{ut-f} yields
\begin{align}\label{xit}
\partial_t\langle\xi, (1-\partial_x^2)f_\lambda\rangle
-&(\dot y-\lambda)\langle \partial_x(\phi_\lambda+\xi),(1-\partial_x^2)f_\lambda\rangle\nonumber \\
&=\langle JS''_\lambda(\phi_\lambda)\xi, (1-\partial_x^2)f_\lambda\rangle
+O(\|\xi\|_{\H1}^2).
\end{align}
Next, we expand the terms in \eqref{xit}. Since $J=-(1-\partial_x^2)^{-1}\partial_x$ is skew-adjoint, we have
$$\langle JS''_\lambda(\phi_\lambda)\xi, (1-\partial_x^2)f_\lambda\rangle
=-\langle \partial_xS''_\lambda(\phi_\lambda)\xi,f_\lambda\rangle.$$
Thus \eqref{xit} becomes
\begin{align}\label{xit'}
\partial_t\langle\xi, (1-\partial_x^2)f_\lambda\rangle
-&(\dot y-\lambda)\langle \partial_x(\phi_\lambda+\xi),(1-\partial_x^2)f_\lambda\rangle\nonumber \\
&=-\langle \partial_xS''_\lambda(\phi_\lambda)\xi,f_\lambda\rangle
+O(\|\xi\|_{\H1}^2).
\end{align}
Furthermore, applying integration by parts, we obtain
\begin{align}
-&(\dot y-\lambda)\langle \partial_x(\phi_\lambda+\xi),(1-\partial_x^2)f_\lambda\rangle\nonumber \\
&=-(\dot y-\lambda)\langle \partial_x(\phi_\lambda+\xi),(1-\partial_x^2)\rho_R(x)
\int_0^x\Gamma_\lambda(z)\,dz\rangle\nonumber \\
&=(\dot y-\lambda)\langle (1-\partial_x^2)(\phi_\lambda+\xi),\Gamma_\lambda\rangle
-(\dot y-\lambda)\langle \partial_x(\phi_\lambda+\xi),(1-\partial_x^2)(\rho_R(z)-1)\int_0^x\Gamma_\lambda(z)\,dz\rangle\nonumber \\
&=(\dot y-\lambda)\langle (1-\partial_x^2)\phi_\lambda,\Gamma_\lambda\rangle
-(\dot y-\lambda)\langle (1-\partial_x^2)\partial_x\phi_\lambda,(\rho_R(z)-1)\int_0^x\Gamma_\lambda(z)\,dz\rangle
+\mathcal{R}_1,\nonumber
\end{align}
where the remainder $\mathcal{R}_1$ is given by
$$\mathcal{R}_1=(\dot y-\lambda)\langle\xi,(1-\partial_x^2)\Gamma_\lambda\rangle
-(\dot y-\lambda)\langle \partial_x\xi,(1-\partial_x^2)(\rho_R(z)-1)\int_0^x\Gamma_\lambda(z)\,dz\rangle.$$
Finally, applying the identities \eqref{xphi}, \eqref{pphi} and \eqref{cphi},  together with a straightforward calculation, yields
\begin{align}
\langle (1-\partial_x^2)\phi_\lambda,\Gamma_\lambda\rangle
=&\langle (1-\partial_x^2)\phi_\lambda,\lambda^3\partial_\lambda\phi_\lambda-\frac{\lambda^2}{p-1}\phi_\lambda+\frac{\lambda}{2}x\partial_x\phi_\lambda+\lambda\phi_\lambda\rangle\nonumber \\
=&\frac{\lambda^3}{2}\partial_\lambda\|\phi_\lambda\|_2^2-\frac{\lambda^2}{p-1}\|\phi_\lambda\|_2^2
-\frac{\lambda}{4}\|\phi_\lambda\|_2^2+\lambda\|\phi_\lambda\|_2^2\nonumber \\
&+\frac{\lambda^3}{2}\partial_\lambda\|\partial_x\phi_\lambda\|_2^2-\frac{\lambda^2}{p-1}\|\partial_x\phi_\lambda\|_2^2
+\frac{\lambda}{4}\|\partial_x\phi_\lambda\|_2^2+\lambda\|\partial_x\phi_\lambda\|_2^2\nonumber \\
=&\frac{8(p^2+1)\lambda^2-2(7p+1)(p-1)\lambda+5(p-1)^2}{4(p+3)(p-1)(\lambda-1)}\|\phi_\lambda\|_2^2\nonumber \\
=&B(\lambda).\label{xit-1}
\end{align}

From the estimate $|\dot y-\lambda|=O(\|\xi\|_{\H1})$ in Proposition \ref{modulation}, together with the fact that $\int_0^x\Gamma_\lambda(z)\,dz\in L^\infty(\R)$, we have
\begin{align}
|(\dot y-\lambda)\langle (1-\partial_x^2)\partial_x\phi_\lambda,&(\rho_R(x)-1)\int_0^x\Gamma_\lambda(z)\,dz\rangle|\nonumber\\
&\lesssim\|\xi\|_{\H1}^2+\int_{|x|>R}\big|(1-\partial_x^2)\partial_x\phi_\lambda\big|^2\,dx\nonumber\\
&= O(\|\xi\|_{\H1}^2+\frac{1}R)\label{xit-2}.
\end{align}
Moreover, we have
\begin{align}
\big| \mathcal{R}_1\big|\lesssim\|\xi\|_{\H1}^2.\label{xit-3}
\end{align}
Combining \eqref{xit-1}, \eqref{xit-2} and \eqref{xit-3}, we obtain
\begin{align}\label{xit4}
-(\dot y-\lambda)\langle \partial_x(\phi_\lambda+\xi),(1-\partial_x^2)f_\lambda\rangle
=B(\lambda)(\dot y-\lambda)+O(\|\xi\|_{\H1}^2+\frac{1}R)
\end{align}

Using integration by parts and the orthogonality condition $\langle\xi, \tau_\lambda\rangle=0$  in \eqref{orth-condition}, we obtain
\begin{align}
-\langle \partial_xS''_\lambda(\phi_\lambda)\xi,f_\lambda\rangle
&=-\langle \partial_xS''_\lambda(\phi_\lambda)\xi,\rho_R(x)\int_0^x\Gamma_\lambda(z)\,dz\rangle\nonumber\\
&=-\langle \partial_xS''_\lambda(\phi_\lambda)\xi,\int_0^x\Gamma_\lambda(z)\,dz\rangle
+\mathcal{R}_2\nonumber\\
&=\langle\xi, S''_\lambda(\phi_\lambda)\Gamma_\lambda\rangle+\mathcal{R}_2\nonumber\\
&=\langle\xi, \tau_\lambda\rangle+\mathcal{R}_2\nonumber\\
&=\mathcal{R}_2,\nonumber
\end{align}
where
$$\mathcal{R}_2
:=-\langle \partial_xS''_\lambda(\phi_\lambda)\xi,(\rho_R(x)-1)\int_0^x\Gamma_\lambda(z)\,dz\rangle.$$
Since $\partial_x[(\rho_R-1)F_\lambda]=\rho_R'F_\lambda+(\rho_R-1)\Gamma_\lambda$, where $F_\lambda(x)=\int_0^x\Gamma_\lambda(z)\,dz$, integration by parts in the distributional duality gives
\[
\mathcal R_2=\big\langle S''_\lambda(\phi_\lambda)\xi,
\rho_R'F_\lambda+(\rho_R-1)\Gamma_\lambda\big\rangle.
\]
The first coefficient is supported in $R\le |x|\le2R$ and is $O(R^{-1})$, while the second one is exponentially small on $|x|\ge R$. Since $S''_\lambda(\phi_\lambda):H^1\to H^{-1}$ is uniformly bounded for $\lambda$ near $c$, Young's inequality yields
\begin{align*}
|\mathcal R_2|
&\lesssim \|\xi\|_{H^1}\left(R^{-1/2}+e^{-\nu R}\right)\\
&\lesssim \|\xi\|_{H^1}^2+\frac{1}R
\end{align*}
for some $\nu>0$, uniformly for $\lambda$ in a sufficiently small neighborhood of $c$. Consequently,
\begin{align}\label{xit5}
  -\langle \partial_xS''_\lambda(\phi_\lambda)\xi,f_\lambda\rangle= O(\|\xi\|_{\H1}^2+\frac{1}R).
\end{align}

Finally, combining \eqref{xit'},\eqref{xit4} and \eqref{xit5}, we obtain
$$\dot y-\lambda=\frac{-1}{B(\lambda)}\partial_t\langle\xi, (1-\partial_x^2)f_\lambda\rangle
  +O(\|\xi\|_{\H1}^2+\frac{1}R).$$

Since $c=c_p^-<0$, the numerator of $B(c)$ is strictly positive,
whereas $c-1<0$. Hence $B(c)<0$. By continuity, there exist
$\delta_B>0$ and $b_0>0$ such that
\[
|B(\lambda)|\ge b_0,
\qquad |\lambda-c|<\delta_B.
\]
Thus, division by $B(\lambda)$ is justified uniformly in this
neighborhood.
\end{proof}

\section{Instability}

\subsection{Virial estimates}
To prove the main theorem, we  first establish the following virial estimates. Let $\varphi\in C^\infty(\mathbb{R})$ be a smooth cutoff function satisfying
\begin{align*}
\varphi(x)=
   \left\{ \aligned
    &x,\quad |x|\leq R,\\
    &0,\quad |x|\geq 2R,
   \endaligned
  \right.
\end{align*}
and $0\leq|\varphi'(x)|\lesssim1$ for any $x\in \R$. $R>0$ is a large parameter to be chosen later.

Define
$$I_1(t)=\int_{\R}\varphi(x-y(t))(\frac{1}{2}u^2+\frac{\kappa}{2} u^{p+1})\,dx.$$

From \eqref{pbbm},  we can rewrite the time evolution as
\begin{align}\label{uth}
u_t
=\partial_x H(u),
\end{align}
where $H(u):=-(1-\partial_x^2)^{-1}[u+\frac{\kappa(p+1)}{2}u^p]$. Then, we also have
\begin{align}\label{uph}
u+\frac{\kappa(p+1)}{2}u^p=-(1-\partial_x^2)H(u)=-H(u)+\partial_x^2H(u).
\end{align}
Recall that the solitary wave profile $\phi_\lambda$ satisfies \eqref{phie}, which can be rewritten as
$$\phi_\lambda=\lambda(1-\partial_x^2)\phi_\lambda-\kappa\frac{p+1}{2}\phi_\lambda^p.$$

The second-order operator $S''_\lambda(\phi_\lambda)$ is given by
$$S_\lambda''(\phi_\lambda)f=\lambda\partial_x^2f+(1-\lambda)f+\kappa\frac{p(p+1)}{2}\phi_\lambda^{p-1}f.$$
Thus, the perturbation $\xi$ can be expressed as
$$\xi=S_\lambda''(\phi_{\lambda})\xi+\lambda(1-\partial_x^2)\xi-\kappa\frac{p(p+1)}{2}\phi_{\lambda}^{p-1}\xi.$$
Combining the above identities, we obtain
\begin{align}\label{phi+xi}
\phi_\lambda+\xi=S_\lambda''(\phi_{\lambda})\xi
+\lambda(1-\partial_x^2)(\phi_\lambda+\xi)
-\kappa\frac{p+1}{2}\phi_\lambda^p-\kappa\frac{p(p+1)}{2}\phi_{\lambda}^{p-1}\xi,
\end{align}
and
\begin{align}\label{hpx}
H(u)=-\lambda(\phi_\lambda+\xi)-(1-\partial_x^2)^{-1}S''_\lambda(\phi_\lambda)\xi-(1-\partial_x^2)^{-1}\mathcal{N}_2(\xi),
\end{align}
where
$$\mathcal{N}_2(\xi):=\kappa\frac{p+1}{2}[(\phi_{\lambda}+\xi)^{p}
-\phi_\lambda^p
-p\phi_{\lambda}^{p-1}\xi]=O(|\xi|^2).$$

We have the following proposition.
\begin{prop}
Assume $u(t,x)=(\phi_\lambda+\xi)(x-y(t))$. Then
\begin{align*}
 I_1'(t)
 = &-(\dot y-\lambda)\big[E(u)+2\lambda\|\partial_x\phi_\lambda\|_2^2\big]-\lambda E(u)\\
 &+\frac{\lambda^2}{2}[\|\phi_\lambda\|_2^2-\|\partial_x\phi_\lambda\|_2^2]
 +O(\|\xi\|_{\H1}^2+\frac{1}R)
\end{align*}

\end{prop}
\begin{proof}
From the definitions of $I_1(t)$ and  $E(u)$,  differentiating $ I_1(t)$ with respect to $t$ yields
\begin{align*}
I'_1(t)=&-\dot y\int_{\R}\varphi'(x-y(t))(\frac{1}{2}u^2+\frac{\kappa}{2} u^{p+1})\,dx
+\int_{\R}\varphi(x-y(t))\partial_t(\frac{1}{2}u^2+\frac{\kappa}{2} u^{p+1})\,dx.\\
=&-\dot y E(u)
-\dot y\int_{\R}\left[\varphi'(x-y(t))-1\right](\frac{1}{2}u^2+\frac{\kappa}{2} u^{p+1})\,dx\\
&+ \int_{\R}\varphi(x-y(t))[u+\frac{\kappa(p+1)}{2} u^{p}]u_t\,dx.
\end{align*}
Using \eqref{uth} and \eqref{uph}, we have
\begin{align*}
\int_{\R}\varphi(x-y(t))&[u+\frac{\kappa(p+1)}{2} u^{p}]u_t\,dx
=\int_{\R}\varphi(x-y(t))[-H(u)+\partial_x^2H(u)]\partial_xH(u)\,dx\\
=&-\frac{1}{2}\int_{\R}\varphi(x-y(t))\partial_x[H^2(u)-(\partial_xH(u))^2]\,dx\\
=&\frac{1}{2}\int_{\R}\varphi'(x-y(t))[H^2(u)-(\partial_xH(u))^2]\,dx\\
=&\frac{1}{2}\int_{\R}H^2(u)\,dx-\frac{1}{2}\int_{\R}(\partial_xH(u))^2\,dx\\
&+\frac{1}{2}\int_{\R}[\varphi'(x-y(t))-1][H^2(u)-(\partial_xH(u))^2]\,dx
\end{align*}
Combining the above identities, we obtain
\begin{align}\label{i1t}
I'_1(t)=-\dot yE(u)+\frac{1}{2}\int_{\R}H^2(u)\,dx-\frac{1}{2}\int_{\R}(\partial_xH(u))^2\,dx+\mathcal{R}_3,
\end{align}
where
$$\mathcal{R}_3:=-\dot y\int_{\R}\left[\varphi'(x-y(t))-1\right](\frac{1}{2}u^2+\frac{\kappa}{2} u^{p+1})\,dx
+\frac{1}{2}\int_{\R}[\varphi'(x-y(t))-1][H^2(u)-(\partial_xH(u))^2]\,dx.$$

We now expand the terms in \eqref{i1t} separately. Applying \eqref{hpx}, we derive
  \begin{align}
\int_{\R}H^2(u)\,dx=&\int_{\R}\left[-\lambda(\phi_\lambda+\xi)
-(1-\partial_x^2)^{-1}S''_\lambda(\phi_\lambda)\xi-(1-\partial_x^2)^{-1}\mathcal{N}_2(\xi)\right]^2\,dx\nonumber\\
=&\lambda^2\|\phi_\lambda+\xi\|_2^2+2\lambda\langle\phi_\lambda, (1-\partial_x^2)^{-1}S''_\lambda(\phi_\lambda)\xi\rangle
+2\lambda\langle\xi,(1-\partial_x^2)^{-1}S''_\lambda(\phi_\lambda)\xi\rangle\nonumber\\
&+\|(1-\partial_x^2)^{-1}S''_\lambda(\phi_\lambda)\xi\|_2^2+O(\|\xi\|_{\H1}^2).\label{hu}
\end{align}
First, we estimate the first term in \eqref{hu}. Expanding gives
\begin{align}\label{hu1}
\|\phi_\lambda+\xi\|_2^2=\|\phi_\lambda\|_2^2+2\langle\xi,\phi_\lambda\rangle+O(\|\xi\|_{\H1}^2).
\end{align}
Furthermore, from \eqref{ut6}, we recall the identity
\begin{align}\label{dy-l}
\langle \xi, S''_\lambda(\phi_\lambda)(1-\partial_x^2)^{-1}\partial_x^2\phi_\lambda\rangle
=-(\dot y-\lambda)\|\partial_x\phi_\lambda\|_2^2+O(\|\xi\|_{\H1}^2).
\end{align}
Using the identity $\phi_\lambda=(1-\partial_x^2)\phi_\lambda+\partial^2_x\phi_\lambda$,
we estimate the second term in \eqref{hu} as follows:
\begin{align}
\langle\phi_\lambda&,(1-\partial_x^2)^{-1}S''_\lambda(\phi_\lambda)\xi\rangle
=\langle\xi,S''_\lambda(\phi_\lambda)(1-\partial_x^2)^{-1}\phi_\lambda\rangle\nonumber\\
=&\langle\xi,S''_\lambda(\phi_\lambda)(1-\partial_x^2)^{-1}(1-\partial_x^2)\phi_\lambda\rangle
+\langle\xi,S''_\lambda(\phi_\lambda)(1-\partial_x^2)^{-1}\partial^2_x\phi_\lambda\rangle\nonumber\\
=&\langle\xi,S''_\lambda(\phi_\lambda)\phi_\lambda\rangle
+\langle\xi,S''_\lambda(\phi_\lambda)(1-\partial_x^2)^{-1}\partial^2_x\phi_\lambda\rangle\nonumber\\
=&\langle\xi,S''_\lambda(\phi_\lambda)\phi_\lambda\rangle
-(\dot y-\lambda)\|\partial_x\phi_\lambda\|_2^2+O(\|\xi\|_{\H1}^2).\label{hu2}
\end{align}
Next, we estimate the fourth term in \eqref{hu}. Using the explicit form of $S''_\lambda(\phi_\lambda)$, we have
\begin{align}
\big\|(1-\partial_x^2)^{-1}S''_\lambda(\phi_\lambda)\xi\big\|_2^2
=&\big\|(1-\partial_x^2)^{-1}\big[-\lambda(1-\partial_x^2)\xi+\xi+\frac{\kappa p(p+1)}{2}\phi_\lambda^{p-1}\xi\big]\big\|_2^2\nonumber\\
=&\|-\lambda\xi+(1-\partial_x^2)^{-1}\big[\xi
+\frac{\kappa p(p+1)}{2}\phi_\lambda^{p-1}\xi\big]\|_2^2\nonumber\\
\lesssim&\|\xi\|_2^2+\|\xi\|_2^2+\|\frac{\kappa p(p+1)}{2}\phi_\lambda^{p-1}\|_{L^\infty}\|\xi\|_2^2\nonumber\\
\lesssim &\|\xi\|_2^2.\label{hu3}
\end{align}
Finally, we estimate the third term. By H\"{o}lder's inequality and \eqref{hu3}, we obtain
\begin{align}\label{hu4}
\left|\lambda\langle\xi,(1-\partial_x^2)^{-1}S''_\lambda(\phi_\lambda)\xi\rangle\right|
\lesssim\|\xi\|_2\,
\|(1-\partial_x^2)^{-1}S''_\lambda(\phi_\lambda)\xi\|_2
\lesssim \|\xi\|_2^2.
\end{align}
Combining \eqref{hu1}-\eqref{hu4}, we obtain
\begin{align}\label{hu5}
\frac{1}{2}\int_{\R}H^2(u)\,dx
&=\frac{\lambda^2}{2}\|\phi_\lambda\|_2^2+\lambda^2\langle\xi,\phi_\lambda\rangle
+\lambda\langle\xi,S''_\lambda(\phi_\lambda)\phi_\lambda\rangle\nonumber\\
&-\lambda(\dot y-\lambda)\|\partial_x\phi_\lambda\|_2^2+O(\|\xi\|_{\H1}^2).
\end{align}

We now turn to the estimate of $\int_{\R}(\partial_xH(u))^2\,dx$. From \eqref{hpx}, \eqref{dy-l} and integration by parts , we have
  \begin{align}
\frac{1}{2}\int_{\R}&(\partial_xH(u))^2\,dx
=\frac{1}{2}\int_{\R}\big[-\lambda(\partial_x\phi_\lambda+\partial_x\xi)
-\partial_x(1-\partial_x^2)^{-1}S''_\lambda(\phi_\lambda)\xi-\partial_x(1-\partial_x^2)^{-1}\mathcal{N}_2(\xi)\big]^2\,dx\nonumber\\
=&\frac{\lambda^2}{2}\|\partial_x\phi_\lambda+\partial_x\xi\|_2^2+\lambda\langle\partial_x\phi_\lambda, \partial_x(1-\partial_x^2)^{-1}S''_\lambda(\phi_\lambda)\xi\rangle
+\lambda\langle\partial_x\xi,\partial_x(1-\partial_x^2)^{-1}S''_\lambda(\phi_\lambda)\xi\rangle\nonumber\\
&+\frac{1}{2}\|\partial_x(1-\partial_x^2)^{-1}S''_\lambda(\phi_\lambda)\xi\|_2^2
+O(\|\xi\|_{\H1}^2)\nonumber\\
=&\frac{\lambda^2}{2}\|\partial_x\phi_\lambda\|_2^2-\lambda^2\langle\xi,\partial_x^2\phi_\lambda\rangle
-\lambda\langle\xi, S''_\lambda(\phi_\lambda)(1-\partial_x^2)^{-1}\partial_x^2\phi_\lambda\rangle
+O(\|\xi\|_{\H1}^2)\nonumber\\
=&\frac{\lambda^2}{2}\|\partial_x\phi_\lambda\|_2^2-\lambda^2\langle\xi,\partial_x^2\phi_\lambda\rangle
+\lambda(\dot y-\lambda)\|\partial_x\phi_\lambda\|_2^2+O(\|\xi\|_{\H1}^2).\label{xhu}
\end{align}

We now  estimate the remainder $\mathcal{R}_3$. From the properties of the cutoff function $\varphi$, we have
\begin{align*}
\big|\mathcal{R}_3\big|
\leq&|\dot y|\Big|\int_{|x-y|>R}[\varphi'(x-y(t))-1](\frac{1}{2}u^2+\frac{\kappa}{2} u^{p+1})\,dx\Big|\\
&+\Big|\int_{|x-y|>R}[\varphi'(x-y(t))-1][H^2(u)-(\partial_xH(u))^2]\,dx\Big|\\
\leq&\int_{|x|>R}\big[u^2+|u|^{p+1}+H(u)^2+(\partial_xH(u))^2\big]\,dx.
\end{align*}
Using Proposition \ref{modulation}, \eqref{hu5}, \eqref{xhu} and H\"{o}lder's inequality, we obtain
\begin{align*}
\big|\mathcal{R}_3\big|
\leq&\int_{|x|>R}\big[|\phi_\lambda+\xi|^2+|\phi_\lambda+\xi|^{p+1}+|\phi_\lambda|^2+|\partial_x\phi_\lambda|^2\big]\,dx\\
&+\Big|\int_{|x|>R}\xi\cdot(\phi_\lambda+\partial_x^2\phi_\lambda)\,dx\Big|.
\end{align*}
Since $\phi_\lambda$ and $\partial_x^2\phi_\lambda$ decay exponentially, we have
$$\int_{|x|>R}\big[|\phi_\lambda|^2+|\partial_x\phi_\lambda|^2\big]\,dx\lesssim \int_{|x|>\R}e^{-C|x|}\,dx\lesssim\frac{1}R.
$$
Hence,
\begin{align}\label{r3}
\big|\mathcal{R}_3\big|
= O(\|\xi\|_{\H1}^2+\frac{1}R).
\end{align}

Substituting \eqref{hu5}, \eqref{xhu} and \eqref{r3} into \eqref{i1t}, we obtain
\begin{align*}
I'_1(t)=&-\dot yE(u)-2\lambda(\dot y-\lambda)\|\partial_x\phi_\lambda\|_2^2
+\frac{\lambda^2}{2}\big(\|\phi_\lambda\|_2^2-\|\partial_x\phi_\lambda\|_2^2\big)\\
&+\langle\xi, \lambda^2\phi_\lambda+\lambda^2\partial_x^2\phi_\lambda+\lambda S''_\lambda(\phi_\lambda)\phi_\lambda\rangle+O(\|\xi\|_{\H1}^2+\frac{1}R).
\end{align*}
Recalling the definition $\tau_\lambda=S''(\phi_\lambda)\Gamma_\lambda=\lambda^2\phi_\lambda+\lambda^2\partial_x^2\phi_\lambda+\lambda S''_\lambda(\phi_\lambda)\phi_\lambda$
and using the orthogonality condition \eqref{orth-condition}, we finally obtain
  \begin{align}
I'_1(t)=&-(\dot y-\lambda)\big[E(u)+2\lambda\|\partial_x\phi_\lambda\|_2^2\big]
-\lambda E(u)
+\frac{\lambda^2}{2}\big(\|\phi_\lambda\|_2^2-\|\partial_x\phi_\lambda\|_2^2\big)\nonumber\\
&+O(\|\xi\|_{\H1}^2+\frac{1}R).\label{I1't}
\end{align}
This proves the proposition.
\end{proof}

From \eqref{xit-yt}, we know that $|\dot y-\lambda|=O(\|\xi\|_{\H1})$. To cancel the first-order term in $I'_1(t)$, we introduce another virial quantity $I_2(t)$. With $f_\lambda$ as defined in \eqref{f}, we set
\begin{align}\label{I2}
  I_2(t):=-\frac{D(\lambda)}{B(\lambda)}\langle\xi, (1-\partial_x^2)f_\lambda\rangle,
\end{align}
where
\[
D(\lambda):=E(\phi_\lambda)+2\lambda\|\partial_x\phi_\lambda\|_2^2
=\frac{4p\lambda-3p+3}{2(p+3)}\|\phi_\lambda\|_2^2.
\]
Differentiating \eqref{I2} with respect to $t$ and using $\dot \lambda=O(\|\xi\|_{\H1})$, we obtain
\begin{align}\label{I2't}
I'_2(t)=-\dot\lambda\partial_\lambda\left[\frac{D(\lambda)}{B(\lambda)}\right]\langle\xi, (1-\partial_x^2)f_\lambda\rangle
-\frac{D(\lambda)}{B(\lambda)}\partial_t\langle\xi, (1-\partial_x^2)f_\lambda\rangle.
\end{align}

From Corollary \ref{bl}, we know that
\begin{align*}
  \dot y-\lambda=\frac{-1}{B(\lambda)}\partial_t\langle\xi, (1-\partial_x^2)f_\lambda\rangle
  +O(\|\xi\|_{\H1}^2+\frac{1}R).
\end{align*}
Moreover, we have
$$E(u)=E(\phi_\lambda+\xi)=E(\phi_\lambda)+\langle E'(\phi_\lambda),\xi\rangle+O(\|\xi\|_{\H1}^2).$$
Combining \eqref{I1't}, \eqref{I2't} and Corollary \ref{bl}, we have
\begin{align}\label{I1+I2}
I'_1(t)+I'_2(t)=-\lambda E(u)+\frac{\lambda^2}{2}\big[\|\phi_\lambda\|_2^2-\|\partial_x\phi_\lambda\|_2^2\big]
+O(\|\xi\|_{\H1}^2+\frac{1}R).
\end{align}


\subsection{Proof of the theorem}
In this section, we fix $c=c_p^-$, set $u_0=(1-a)\phi_c,\ 0<a\ll1$, and assume $u\in U_\varepsilon(\phi_c)$. It then follows that $|\lambda-c|\lesssim \varepsilon$. Define
$$I(t):=I_1(t)+I_2(t),$$
that is,
\begin{align}
I(t)=\int_{\R}\varphi(x-y(t))(\frac{1}{2}u^2+\frac{\kappa}{2} u^{p+1})\,dx
-\frac{D(\lambda)}{B(\lambda)}\langle\xi, (1-\partial_x^2)f_\lambda\rangle.
\end{align}
Consequently, $I(t)$ is uniformly bounded:
\begin{align}\label{I-B}
  \sup_{t\in\R}I(t)\lesssim R(\|\phi_c\|_2^2+1).
\end{align}

From \eqref{I1+I2} and the conservation of energy $E(u)=E(u_0)$, we have
\begin{align*}
I'(t)=-\lambda[E(u_0)-E(\phi_c)]
-\lambda E(\phi_c)+\frac{\lambda^2}{2}\big[\|\phi_\lambda\|_2^2-\|\partial_x\phi_\lambda\|_2^2\big]
+O(\|\xi\|_{\H1}^2+\frac{1}R).
\end{align*}
We now estimate $I'(t)$. For the first term, using \eqref{pphi}, we obtain
\begin{align}\label{Eu0-phic}
-\lambda[E(u_0)-E(\phi_c)]
=&-\lambda\langle E'(\phi_c),u_0-\phi_c\rangle+O(\|u_0-\phi_c\|_{\H1}^2)\nonumber\\
=&\lambda\langle\phi_c+\frac{\kappa(p+1)}{2}\phi_c^p, a\phi_c\rangle+O(a^2)\nonumber\\
=&a\lambda\big[\|\phi_c\|_2^2+\frac{\kappa(p+1)}{2}\int_{\R}\phi_c^{p+1}dx\big]+O(a^2)\nonumber\\
=&a\lambda\frac{2(p+1)c-(p-1)}{p+3}\|\phi_c\|_2^2+O(a^2)\nonumber\\
=&ac\frac{2(p+1)c-(p-1)}{p+3}\|\phi_c\|_2^2+O(a|\lambda-c|)+O(a^2)\nonumber\\
\geq&\frac{C_1}{2}a.
\end{align}
where $C_1:=\frac{2(p+1)c^2-(p-1)c}{p+3}\|\phi_c\|_2^2>0$.

Moreover, by \eqref{xphi}, we obtain
\begin{align*}
g(\lambda):=&-\lambda E(\phi_c)+\frac{\lambda^2}{2}\big[\|\phi_\lambda\|_2^2-\|\partial_x\phi_\lambda\|_2^2\big]\\
=&-\lambda\frac{4c+p-1}{2(p+3)}\|\phi_c\|_2^2+\lambda\frac{4\lambda+p-1}{2(p+3)}\|\phi_\lambda\|_2^2.
\end{align*}
At $\lambda=c$, we have
\begin{align}
g(c)=0.
\end{align}
From \eqref{cphi}, a direct computation yields
\begin{align}
\partial_\lambda\|\phi_\lambda\|_2^2=\frac{4\lambda-p+1}{2(p-1)(\lambda-1)\lambda}\|\phi_\lambda\|_2^2\label{phi-1},
\end{align}
and
\begin{align}
\partial^2_\lambda\|\phi_\lambda\|_2^2
=\frac{-8(p-3)\lambda^2+4(p-1)(p-3)\lambda-(p-1)^2}{4(p-1)^2(\lambda-1)^2\lambda^2}\|\phi_\lambda\|_2^2.\label{phi-2}
\end{align}
Differentiating $g(\lambda)$ with respect to $\lambda$, we have
\begin{align*}
  g'(\lambda)=&-\frac{4c+p-1}{2(p+3)}\|\phi_c\|_2^2
  +\frac{8\lambda+p-1}{2(p+3)}\|\phi_\lambda\|_2^2
  +\frac{4\lambda^2+(p-1)\lambda}{2(p+3)}\partial_\lambda\|\phi_\lambda\|_2^2.
\end{align*}
Using $c=c_p^-$, \eqref{phi-1}, and Lemma \ref{Qc}, we obtain
\begin{align}
g'(c)=\frac{8(p+1)c^2-8(p-1)c-(p-1)^2}
{4(p+3)(p-1)(c-1)}\|\phi_c\|_2^2=0.
\end{align}
Differentiating once more with respect to $\lambda$, we obtain
\begin{align*}
g''(\lambda)
=&\frac{4}{p+3}\|\phi_\lambda\|_2^2
+\frac{8\lambda+p-1}{(p+3)}\partial_\lambda\|\phi_\lambda\|_2^2
+\frac{4\lambda^2+(p-1)\lambda}{2(p+3)}\partial^2_\lambda\|\phi_\lambda\|_2^2.
\end{align*}
Substituting \eqref{phi-1} and \eqref{phi-2} into $g''(\lambda)$ and then setting $c=c_p^-$, we obtain
\begin{align*}
g''(c)=&\frac{\|\phi_c\|_2^2}{8(p+3)(p-1)^2(c-1)^2c}
\Big[32p(p+1)c^3-72(p+1)(p-1)c^2\\
&\hspace{42mm}+36(p-1)^2c+3(p-1)^3\Big].
\end{align*}
To make the sign transparent, use the endpoint identity \eqref{zero}. The polynomial in brackets reduces to
\[
\frac{2(p-1)^2(p+3)}{p+1}\big[2c(p-3)-p+1\big].
\]
For $p=2$, the last factor equals $-2c-1<0$ at $c=c_2^-$; for $p=3$ it equals $-2$; and for $p>3$ it is also strictly negative because $c<0$. Since the denominator of the preceding expression for $g''(c)$ is negative, we conclude that $g''(c)>0$.

Therefore, Taylor's theorem yields
\begin{align}\label{glamda}
g(\lambda)=&g(c)+g'(c)(\lambda-c)+\frac{1}{2}g''(c)(\lambda-c)^2+o\big((\lambda-c)^2\big)\nonumber\\
\geq&C_2(\lambda-c)^2,
\end{align}
where $C_2:=\frac{g''(c)}{2}>0$.

It remains to estimate the last term in $I'(t)$. Since $u(t,x)=(\phi_\lambda+\xi)(x-y(t))$ and $S'_\lambda(\phi_\lambda)=0$,  we have
\begin{align*}
S_\lambda(u)-S_\lambda(\phi_\lambda)
=&\langle S'_\lambda(\phi_\lambda), \xi\rangle+\frac{1}{2}\langle S''_\lambda(\phi_\lambda)\xi, \xi\rangle
+o(\|\xi\|^2_{\H1})\\
=&\frac{1}{2}\langle S''_\lambda(\phi_\lambda)\xi, \xi\rangle
+o(\|\xi\|^2_{\H1}).
\end{align*}
By Corollary \ref{uniform-coercivity} and the uniform Taylor expansion
of $S_\lambda$ around $\phi_\lambda$, after reducing $\varepsilon$ if
necessary,
\begin{equation}\label{geqxi}
S_\lambda(u)-S_\lambda(\phi_\lambda)
\ge \frac{C_*}{4}\|\xi\|_{H^1}^2.
\end{equation}

On the other hand, we write
\begin{align}\label{sup}
S_\lambda(u)-S_\lambda(\phi_\lambda)=S_\lambda(u_0)-S_\lambda(\phi_c)
+S_\lambda(\phi_c)-S_\lambda(\phi_\lambda).
\end{align}
We now estimate each term in \eqref{sup}. Using the definitions of $E$, $S$ and  the fact that $S'_c(\phi_c)=0$, we have
\begin{align}
S_\lambda(u_0)-S_\lambda(\phi_c)
=&S_c(u_0)-S_c(\phi_c)+(c-\lambda)[Q(u_0)-Q(\phi_c)]\nonumber\\
=&\langle S'_c(\phi_c), u_0-\phi_c\rangle
+(c-\lambda)\langle Q'(\phi_c),u_0-\phi_c\rangle+O(\|u_0-\phi_c\|_{\H1}^2)\nonumber\\
=&a(\lambda-c)\langle \phi_c-\partial_x^2\phi_c, \phi_c\rangle +O(a^2)\nonumber\\
=&a(\lambda-c)\big[1+\frac{(p-1)(c-1)}{(p+3)c}\big]\|\phi_c\|_2^2
+O(a^2+a|\lambda-c|).\label{u0-phi}
\end{align}

Note that $\phi_\lambda-\phi_c=(\lambda-c)\partial_c\phi_c+O\big((\lambda-c)^2\big)$. By Lemma \ref{S''} and Lemma \ref{Qc}, we have
\begin{align}
S_\lambda(\phi_\lambda)-S_\lambda(\phi_c)
=&S_c(\phi_\lambda)-S_c(\phi_c)+(c-\lambda)[Q(\phi_\lambda)-Q(\phi_c)]\nonumber\\
=&\langle S'_c(\phi_c), \phi_\lambda-\phi_c\rangle
+\frac{1}{2}\langle S''_c(\phi_c)(\phi_\lambda-\phi_c), \phi_\lambda-\phi_c\rangle+o((\lambda-c)^2)\nonumber\\
&+(c-\lambda)\langle Q'(\phi_c),\phi_\lambda-\phi_c\rangle+(c-\lambda)O(\|\phi_\lambda-\phi_c\|_{\H1}^2)\nonumber\\
=&\frac{1}{2}(\lambda-c)^2\langle S''_c(\phi_c)\partial_c\phi_c, \partial_c\phi_c\rangle
-(\lambda-c)^2\langle Q'(\phi_c),\partial_c\phi_c\rangle+o((\lambda-c)^2)\nonumber\\
=&o((\lambda-c)^2).\label{phl-phc}
\end{align}

Combining \eqref{geqxi}, \eqref{u0-phi} and \eqref{phl-phc}, we obtain
\begin{align}\label{xi-xi-2}
\|\xi\|^2_{\H1}= O(a^2+a|\lambda-c|)+o((\lambda-c)^2).
\end{align}

Substituting \eqref{Eu0-phic}, \eqref{glamda}, \eqref{xi-xi-2} into $I'(t)$,  we obtain
\begin{align*}
I'(t)\geq C_1a+C_2(\lambda-c)^2+O(a^2+a|\lambda-c|)+o((\lambda-c)^2)+O(\frac{1}R).
\end{align*}
Choosing $R$ satisfying $\frac{1}R\leqslant a^2$, $\varepsilon$ and $a_0$ small enough, for any $a\in(0,a_0)$, we have
$$I'(t)\geq \frac{1}{4}C_1a+\frac{1}{4}C_2(\lambda-c)^2 \geq \frac{1}{4}C_1a>0, \qquad t\geq0 .$$
Hence $I(t)\to+\infty$ as $t\to\infty$, contradicting the uniform
bound \eqref{I-B}. This contradiction proves the orbital instability
of $\phi_c$ and completes the proof of Theorem \ref{thm:mainthm}.


\end{document}